\documentclass[10pt]{amsart}
\usepackage{graphicx,amssymb}
\usepackage{a4wide}

\theoremstyle{plain}
\newtheorem{theorem}{Theorem}[section]
\newtheorem{lemma}[theorem]{Lemma}
\newtheorem{corollary}[theorem]{Corollary}

\theoremstyle{definition}
\newtheorem{definition}[theorem]{Definition}
\newtheorem{example}[theorem]{Example}
\newenvironment{claim}
{\par\medskip\noindent\textbf{Claim.}\ \ \ignorespaces}{\par}
\newenvironment{remark}
{\par\medskip\noindent\textbf{Remark.}\ \ \ignorespaces}{\par\medskip}

\def\arA{\mathbf A}
\def\arB{\mathbf B}
\def\arC{\mathbf C}
\def\C{\mathcal C}
\def\n{\mathbf n}
\def\lk{\operatorname{lk}}
\def\R{\operatorname{\mathcal R}}
\def\ext{{\operatorname{ext}}}
\def\Rext{\R_\ext}
\def\myangle#1{\langle #1\rangle}

\begin{document}

\title[Uniqueness of roots for some Artin groups]
{Uniqueness of roots up to conjugacy\\
for some affine and finite type Artin groups}

\author{Eon-Kyung Lee and Sang-Jin Lee}

\address{Department of Mathematics, Sejong University,
Seoul, 143-747, Korea}
\email{eonkyung@sejong.ac.kr}

\address{Department of Mathematics, Konkuk University,
Seoul, 143-701, Korea}
\email{sangjin@konkuk.ac.kr}

\date{\today}

\begin{abstract}
Let $G$ be one of the Artin groups
of finite type ${\mathbf B}_n={\mathbf C}_n$
and affine type $\tilde{\mathbf A}_{n-1}$, $\tilde{\mathbf C}_{n-1}$.
In this paper, we show that if $\alpha$ and $\beta$ are elements of $G$
such that $\alpha^k=\beta^k$ for some nonzero integer $k$,
then $\alpha$ and $\beta$ are conjugate in $G$.
For the Artin group of type $\mathbf A_n$, this was recently
proved by J.~Gonz\'alez-Meneses.

In fact, we prove a stronger theorem,
from which the above result follows easily
by using descriptions of those Artin groups
as subgroups of the braid group on $n+1$ strands.
Let $P$ be a subset of $\{1,\ldots,n\}$.
An $n$-braid is said to be \emph{$P$-pure} if its induced permutation
fixes each $i\in P$, and \emph{$P$-straight} if it is $P$-pure and
it becomes trivial when we delete all the $i$-th strands for $i\not\in P$.
Exploiting the Nielsen-Thurston classification of braids,
we show that if $\alpha$ and $\beta$ are $P$-pure $n$-braids
such that $\alpha^k=\beta^k$ for some nonzero integer $k$,
then there exists a $P$-straight $n$-braid $\gamma$
with $\beta=\gamma\alpha\gamma^{-1}$.
Moreover, if $1\in P$, the conjugating element $\gamma$
can be chosen to have the first strand algebraically unlinked with
the other strands.
Especially in case of $P=\{1,\ldots,n\}$, our result implies
the uniqueness of roots of pure braids,
which was known by V.~G.~Bardakov and by D.~Kim and D.~Rolfsen.

\end{abstract}

\maketitle

\section{Introduction}

Let $M$ be a symmetric $n\times n$ matrix with integer entries
$m_{ij}\in\mathbb N\cup\{\infty\}$ where
$m_{ii}=1$ and $m_{ij}\ge 2$ for $i\ne j$.
The \emph{Artin group}
of type $M$ is defined by the presentation
$$
A(M)=\langle s_1,\ldots,s_n\mid \underbrace{s_is_js_i\cdots}_{m_{ij}}
=\underbrace{s_js_is_j\cdots}_{m_{ij}}
\quad\mbox{for all $i\ne j$, $m_{ij}\ne\infty$}\rangle.
$$
The \emph{Coxeter group} $W(M)$ of type $M$ is the quotient
of $A(M)$ by the relation $s_i^2=1$.
We say that $A(M)$ is of \emph{finite type} if
the associated Coxeter group $W(M)$ is finite,
and that $A(M)$ is of \emph{affine (or Euclidean) type}
if $W(M)$ acts as a proper, cocompact group of isometries on some
Euclidean space with the generators $s_1,\ldots,s_n$ acting
as affine reflections.
It is convenient to define an Artin group by a \emph{Coxeter graph},
whose vertices are numbered $1,\ldots,n$ and which has an edge labelled $m_{ij}$
between the vertices $i$ and $j$ whenever $m_{ij}\ge 3$
or $m_{ij}=\infty$.
The label 3 is usually suppressed.

In this paper, we show the uniqueness of roots
up to conjugacy for elements of the Artin groups of
finite type $\arA_n$, $\arB_n=\arC_n$ and affine type $\tilde \arA_{n-1}$, $\tilde \arC_{n-1}$.
The Coxeter graphs associated to them are in Figure~\ref{fig:graph}.

\begin{figure}[t]
\begin{tabular}{ccccc}
\raisebox{1em}{$\arA_n$} &
\includegraphics[scale=.5]{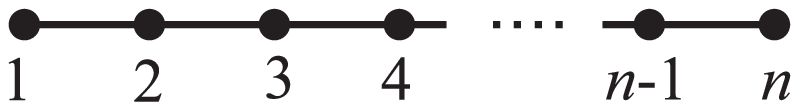} &\mbox{}\qquad\mbox{}&
\raisebox{1em}{$\tilde \arA_{n-1}$} &
\includegraphics[scale=.5]{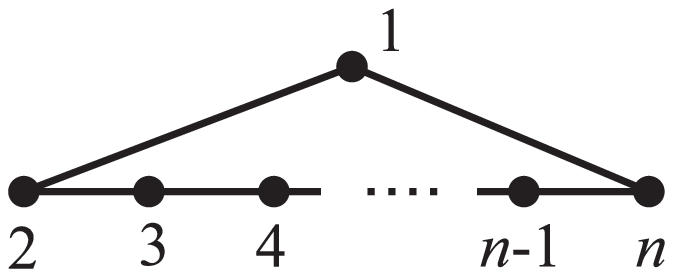}\\
\raisebox{1em}{$\arB_n$} &
\includegraphics[scale=.5]{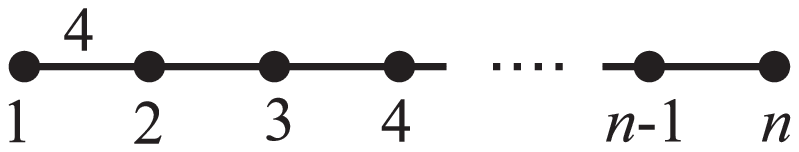} &&
\raisebox{1em}{$\tilde \arC_{n-1}$} &
\includegraphics[scale=.5]{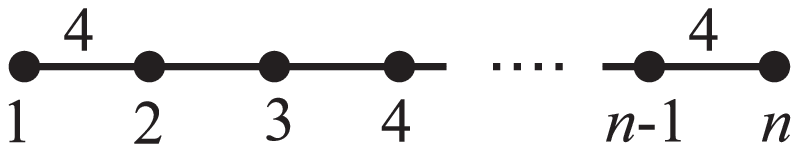}
\end{tabular}
\caption{Coxeter graphs}
\label{fig:graph}
\end{figure}

\begin{theorem}\label{thm:artin}
Let\/ $G$ denote one of the Artin groups of finite type
$\arA_n$, $\arB_n=\arC_n$ and affine type $\tilde \arA_{n-1}$, $\tilde \arC_{n-1}$.
If\/ $\alpha,\beta\in G$ are such that
$\alpha^k=\beta^k$ for some nonzero integer $k$,
then $\alpha$ and $\beta$ are conjugate in $G$.
\end{theorem}

In fact, we prove a stronger theorem.
Before stating it, let us explain the motivations.
The Artin group $A(\arA_n)$ is well-known as $B_{n+1}$,
the braid group on $n+1$ strands.
The generators of $B_{n+1}$ are usually written as $\sigma_i$,
hence it has the presentation
$$
B_{n+1}=\left\langle\sigma_1,\ldots,\sigma_n\biggm|
\begin{array}{ll}
\sigma_i\sigma_j=\sigma_j\sigma_i & \mbox{if } |i-j| > 1, \\
\sigma_i\sigma_{j}\sigma_i=\sigma_{j}\sigma_i\sigma_{j}
& \mbox{if } |i-j| = 1.
\end{array}
\right\rangle.
$$

The following are well-known theorems on
uniqueness of roots of braids.

\begin{theorem}[J.~Gonz\'alez-Meneses~\cite{Gon03}]\label{thm:gon}
Let $\alpha$ and $\beta$ be $n$-braids
such that $\alpha^k=\beta^k$ for some nonzero integer $k$.
Then $\alpha$ and $\beta$ are conjugate in $B_n$.
\end{theorem}

\begin{theorem}[V.~G.~Bardakov~\cite{Bar},
D.~Kim and D.~Rolfsen~\cite{KR03}]\label{thm:KR03}
Let $\alpha$ and $\beta$ be pure braids
such that $\alpha^k=\beta^k$ for some nonzero integer $k$.
Then $\alpha$ and $\beta$ are equal.
\end{theorem}

Theorem~\ref{thm:gon} was conjectured by G.~S.~Makanin~\cite{Mak71}
in the early seventies, and proved recently
by J.~Gonz\'alez-Meneses.
Thus the new contribution of Theorem~\ref{thm:artin}
is for the Artin groups of type $\arB_n=\arC_n$,
$\tilde \arA_{n-1}$ and $\tilde \arC_{n-1}$.
Theorem~\ref{thm:KR03} was first proved by V.~G.~Bardakov
by combinatorial arguments, and it follows easily
from the bi-orderability of pure braids by D.~Kim and D.~Rolfsen.
(To see this, let $<$ be a bi-ordering of pure braids.
If $\alpha>\beta$ (resp. $\alpha<\beta$),
then $\alpha^k>\beta^k$ (resp. $\alpha^k<\beta^k$) for all $k\ge 1$.
Therefore, $\alpha^k=\beta^k$ implies $\alpha=\beta$.)

It is worth mentioning that D.~Bessis showed
the uniqueness of roots up to conjugacy
for periodic elements in the braid groups of
irreducible well-generated complex reflection groups,
and hence for periodic elements in finite type Artin groups:
if $G$ is the braid group of an irreducible well-generated
complex reflection group and if $\alpha,\beta\in G$ are such that
$\alpha$ has a central power and $\alpha^k=\beta^k$ for some nonzero integer $k$,
then $\alpha$ and $\beta$ are conjugate in $G$~\cite[Theorem 12.5~(ii)]{Bes06}.

Comparing the above two theorems, we can see that
one obtains a stronger result for pure braids.
Motivated by the above observation,
we study the case of ``partially pure'' braids---that is,
braids some of whose strands are pure.
Moreover, the Artin groups $A(\arB_n)$, $A(\tilde \arA_{n-1})$
and $A(\tilde \arC_{n-1})$ are isomorphic
to some subgroups of $B_{n+1}$, which can be described
by pure strands and linking number of the first strand
with the other strands.
In order to deal with ``partially pure'' braids and elements of
those Artin groups simultaneously, we introduce the following definitions.

\begin{definition}
For an $n$-braid $\alpha$, let $\pi_\alpha$ denote
the induced permutation of $\alpha$.
\begin{itemize}
\item
For an $n$-braid $\alpha$ and an integer $1\le i\le n$,
we say that $\alpha$ is \emph{$i$-pure},
or the $i$-th strand of $\alpha$ is \emph{pure},
if $\pi_\alpha(i)=i$.
\item
Let $B_{n,1}$ denote the subgroup of $B_n$ consisting
of 1-pure braids.
\item
Let $P\subset\{1,\ldots,n\}$.
An $n$-braid $\alpha$ is said to be
\emph{$P$-pure} if $\alpha$ is $i$-pure for each $i\in P$.
Note that $\{1,\ldots,n\}$-pure braids are
nothing more than pure braids in the usual sense.

\item
Let $P\subset\{1,\ldots,n\}$.
An $n$-braid is said to be \emph{$P$-straight}
if it is  $P$-pure and it becomes trivial
when we delete all the $i$-th strands for $i\not\in P$.
Note the following:
if $|P|=1$, then a braid is $P$-pure if and only if
it is $P$-straight;
if $|P|=n$ and $\alpha$ is a $P$-straight $n$-braid,
then $\alpha$ is the identity;
a braid $\alpha$ is called a \emph{brunnian braid}
if it is $P$-straight for all $P$ with $|P|=n-1$.
\end{itemize}
\end{definition}

For example, the braid in Figure~\ref{fig:p-pure}
is $\{1,4,5\}$-pure, $\{1,4\}$-straight and $\{1,5\}$-straight.

\begin{definition}
There is a homomorphism $\lk:B_{n,1}\to \mathbb Z$ which measures
the linking number of the first strand with the other strands:
let $\sigma_1,\ldots,\sigma_{n-1}$ be the Artin generators for $B_n$,
then $B_{n,1}$ is generated
by $\sigma_1^2,\sigma_2,\sigma_3,\ldots,\sigma_{n-1}$,
and the homomorphism $\lk$ is defined by $\lk(\sigma_1^2)=1$ and
$\lk(\sigma_i)=0$ for $i\ge 2$.
Note that $\lk(\cdot )$ is a conjugacy invariant in $B_{n,1}$
because $\lk(\gamma\alpha\gamma^{-1})=\lk(\gamma)+\lk(\alpha)-\lk(\gamma)=\lk(\alpha)$
for any $\alpha,\gamma\in B_{n,1}$.
A braid $\alpha$ is said to be \emph{1-unlinked}
if it is 1-pure and $\lk(\alpha)=0$.
For example, the braid in Figure~\ref{fig:p-pure}
is 1-unlinked.
\end{definition}

\begin{figure}
$$
\includegraphics[scale=.9]{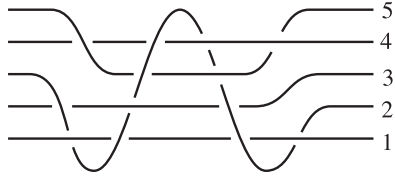}
$$
\caption{This braid is $\{1,4,5\}$-pure, $\{1,4\}$-straight,
$\{1,5\}$-straight and 1-unlinked.}\label{fig:p-pure}
\end{figure}

It is well known that the following isomorphisms hold~\cite{Cri99, All02, CC05, BM05}:
\begin{eqnarray*}
A(\arB_n) &\simeq& B_{n+1,1};\\
A(\tilde \arA_{n-1}) &\simeq&
\{\alpha\in B_{n+1,1}\mid \mbox{$\alpha$ is 1-unlinked} \};\\
A(\tilde \arC_{n-1}) &\simeq&
\{\alpha\in B_{n+1,1}\mid \mbox{$\alpha$ is $\{1,n+1\}$-pure} \}.
\end{eqnarray*}

As we do not need to consider the Artin group $A(\arA_n)$
due to Theorem~\ref{thm:gon},
it suffices to consider 1-pure braids.
From now on, we restrict ourselves to $B_{n,1}$,
the group of 1-pure braids on $n$ strands.
Exploiting the Nielsen-Thurston classification of braids,
we establish the following theorem.

\begin{theorem}\label{thm:main}
Let $P$ be a subset of\/ $\{1,\ldots,n\}$ with $1\in P$.
Let $\alpha$ and $\beta$ be $P$-pure $n$-braids
such that $\alpha^k=\beta^k$ for some nonzero integer $k$.
Then there exists a $P$-straight, 1-unlinked $n$-braid $\gamma$
with $\beta=\gamma\alpha\gamma^{-1}$.
\end{theorem}

Applying Theorem~\ref{thm:main} to $\{1\}$-pure $(n+1)$-braids
(resp. $\{1,n+1\}$-pure $(n+1)$-braids),
we have Theorem~\ref{thm:artin} for $A(\arB_n)$ and $A(\tilde\arA_{n-1})$
(resp. for $A(\tilde\arC_{n-1})$).

\smallskip

We close this section with some remarks.
An easy consequence of Theorem~\ref{thm:artin} is the following.
\begin{quote}
Let $G$ denote one of the Artin groups of finite type
$\arA_n$, $\arB_n=\arC_n$ and affine type $\tilde \arA_{n-1}$, $\tilde \arC_{n-1}$.
Let $\alpha,\beta\in G$ and let $k$ be a nonzero integer.
Then $\alpha$ is conjugate to $\beta$ if and only if\/
$\alpha^k$ is conjugate to $\beta^k$.
\end{quote}

\smallskip

Theorem~\ref{thm:KR03} follows easily from Theorem~\ref{thm:main}:
Let $\alpha$ and $\beta$ be pure $n$-braids with $\alpha^k=\beta^k$.
In our terminology, both $\alpha$ and $\beta$
are $\{1,\ldots,n\}$-pure, hence there exists
a $\{1,\ldots,n\}$-straight $n$-braid $\gamma$ such that
$\beta=\gamma\alpha\gamma^{-1}$.
Because $\gamma$ is $\{1,\ldots,n\}$-straight,
we have $\gamma=1$, hence $\alpha=\beta$.

\smallskip

Theorem~\ref{thm:artin}, even for $A(\arB_n)$, does not follow
easily from Theorem~\ref{thm:gon} because
there are 1-pure braids that are conjugate in $B_n$,
but not in the subgroup $B_{n,1}$, as in the following example.

\begin{example}
Consider the 1-pure 3-braids which are depicted in
Figure~\ref{fig:B31}:
$$
\left\{\begin{array}{l}
\alpha_1=\sigma_1^2,\\
\beta_1=\sigma_2^2,\end{array}\right.
\qquad\mbox{and}\qquad
\left\{\begin{array}{l}
\alpha_2=\sigma_1^2\sigma_2^4,\\
\beta_2=\sigma_2^2\sigma_1^4.
\end{array}\right.
$$
Because $\Delta\alpha_i\Delta^{-1}=\beta_i$ for $i=1,2$,
where $\Delta=\sigma_1\sigma_2\sigma_1$,
the braid $\alpha_i$ is conjugate to $\beta_i$ in $B_3$.
However, $\alpha_i$ is not conjugate to $\beta_i$ in $B_{3,1}$ for $i=1, 2$
because $\lk(\alpha_1)=\lk(\alpha_2)=1$,
$\lk(\beta_1)=0$ and $\lk(\beta_2)=2$.
Note that $\alpha_1$ and $\beta_1$ are reducible,
and that $\alpha_2$ and $\beta_2$ are pseudo-Anosov.
\end{example}

\begin{figure}\center
\begin{tabular}{cccc}
\mbox{}~\includegraphics[scale=1.2]{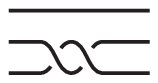}~\mbox{} &
\mbox{}~\includegraphics[scale=1.2]{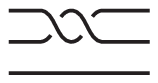}~\mbox{} &
\mbox{}~\includegraphics[scale=1.2]{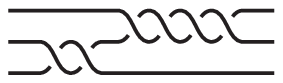}~\mbox{} &
\mbox{}~\includegraphics[scale=1.2]{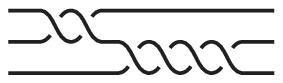}~\mbox{}\\
$\alpha_1=\sigma_1^2$ &
$\beta_1=\sigma_2^2$ &
$\alpha_2=\sigma_1^2\sigma_2^4$ &
$\beta_2=\sigma_2^2\sigma_1^4$
\end{tabular}
\caption{$\alpha_i, \beta_i\in B_{n,1}$ are conjugate in $B_n$ but not in $B_{n,1}$.}
\label{fig:B31}
\end{figure}

\section{Preliminaries}
Here, we review basic definitions and results on braids.
See~\cite{Art25,Bir,Thu88,FLP79,GW04,LL08}. Let
$D^2=\{z\in\mathbb{C}: |z|\le n+1\}$, and let $D_n$ be the
$n$-punctured disk $D^2\setminus\{1,2,\ldots,n\}$. The Artin braid
group $B_n$ is the group of automorphisms of $D_n$ that fix the
boundary pointwise, modulo isotopy relative to the boundary.
Geometrically, an $n$-braid can be interpreted as an isotopy class
of the collections of pairwise disjoint $n$ strands $l=l_1\cup\cdots
\cup l_n\subset D^2\times[0,1]$ such that $l\cap (D^2\times\{t\})$
consists of $n$ points for each $t\in[0,1]$, and, in particular, it
is $\{(1,t),\ldots,(n,t) \}$ for $t\in\{0,1\}$. The admissible
isotopies lie in the interior of $D^2\times [0,1]$. The center of
the $n$-braid group $B_n$ is infinite cyclic generated by
$\Delta^2$, where $\Delta=\sigma_1(\sigma_2\sigma_1)\cdots
(\sigma_{n-1}\cdots\sigma_1)$.

The well-known Nielsen-Thurston classification
of mapping classes of punctured surfaces into periodic,
reducible and pseudo-Anosov ones~\cite{Thu88,FLP79}
yields an analogous classification of braids:
an $n$-braid $\alpha$ is \emph{periodic}
if some power of $\alpha$ is central;
$\alpha$ is \emph{reducible} if there exists an essential curve
system in $D_n$ which is invariant up to isotopy under the action of $\alpha$;
$\alpha$ is \emph{pseudo-Anosov} if no non-trivial power of $\alpha$ is
reducible.

\begin{lemma}\label{lem:NT}
Let $\alpha,\beta\in B_n$ be such that $\alpha^k=\beta^k$
for a nonzero integer $k$. Then
\begin{itemize}
\item[(i)]
$\alpha$ and $\beta$ are of the same Nielsen-Thurston type;
\item[(ii)]
if $\alpha$ is pseudo-Anosov, then $\alpha=\beta$.
\end{itemize}
\end{lemma}

\begin{proof}
(i) is well known.
(ii) was proved by Gonz\'alez-Meneses~\cite{Gon03}.
\end{proof}

\subsection{Periodic braids}

Let $\delta=\sigma_{n-1}\cdots\sigma_1$ and $\epsilon=\delta\sigma_1$,
then $\delta^n = \Delta^2 = \epsilon^{n-1}$.
(If we need to specify the number of strands, we will write $\delta=\delta_{(n)}$,
$\epsilon = \epsilon_{(n)}$ and $\Delta = \Delta_{(n)}$.)
Note that $\delta$ and $\epsilon$ are represented
by rigid rotations of the $n$-punctured disk as in Figure~\ref{fig:circ}
when the punctures are at the center of the disk
or on a round circle centered at the origin.
By Brouwer, Ker\'ekj\'art\'o and Eilenberg,
it is known that
an $n$-braid $\alpha$ is periodic if and only if
it is conjugate to a power of either $\delta$ or
$\epsilon$~\cite{Bro19,Ker19,Eil34,BDM02}.

\begin{figure}
\begin{tabular}{ccc}
\includegraphics[scale=.65]{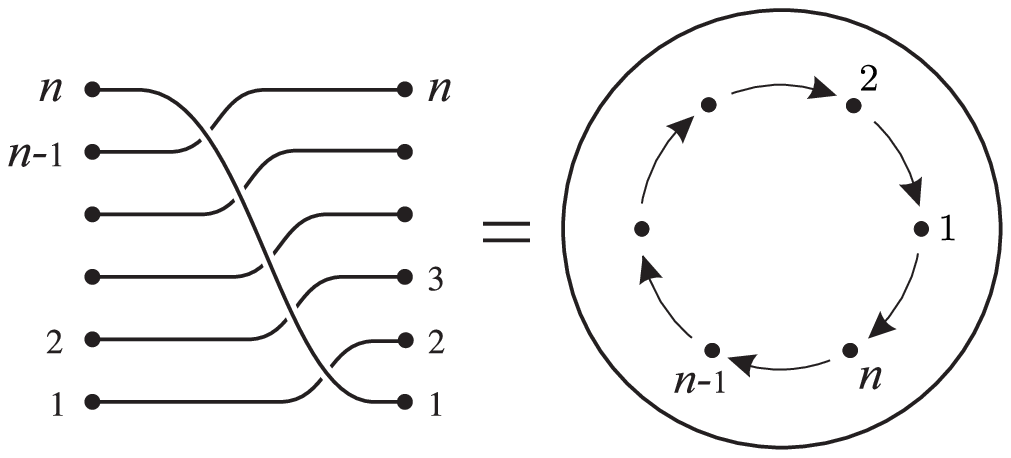}
& &
\includegraphics[scale=.65]{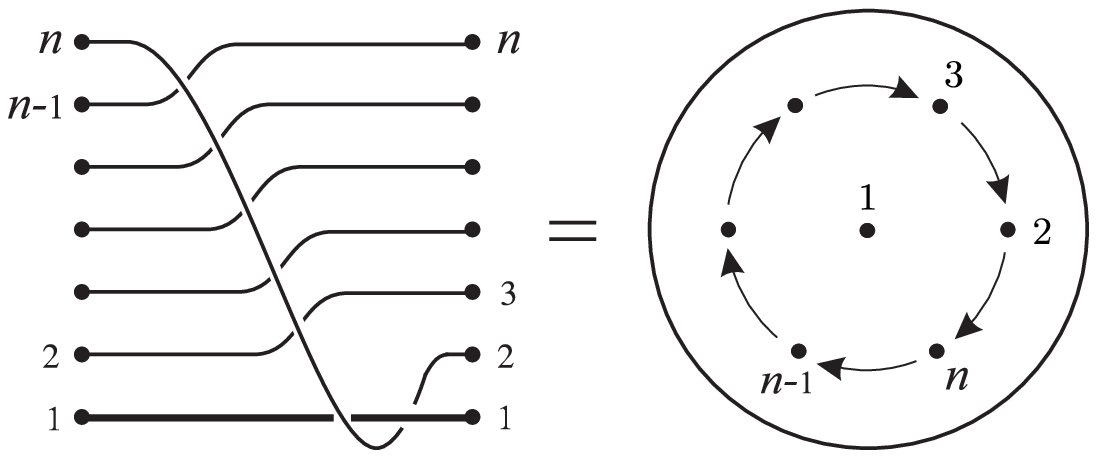}\\
(a) $\delta_{(n)} = \sigma_{n-1}\sigma_{n-2}\cdots\sigma_1 \in B_{n}$
&& (b) $\epsilon_{(n)} = \delta_{(n)}\sigma_1\in B_{n}$
\end{tabular}
\caption{
The braid $\delta_{(n)}$ is represented by the $2\pi/n$-rotation
of the $n$-punctured disk in a clockwise direction
where the punctures lie on a round circle as in (a).
The braid $\epsilon_{(n)}$ is represented by the $2\pi/(n-1)$-rotation
of the $n$-punctured disk in a clockwise direction
where one puncture is at the center
and the other $n-1$ punctures lie on a round circle as in (b).}
\label{fig:circ}
\end{figure}

\begin{lemma}\label{lem:per}
An $n$-braid $\alpha$ is periodic if and only if
$\alpha$ is conjugate to either
$\delta^m$ or $\epsilon^m$ for some integer $m$.
Further, if $\alpha$ is periodic and non-central,
then exactly one of the following holds.
\begin{itemize}
\item[(i)]
$\alpha$ is conjugate to $\delta^m$ for some $m\not\equiv 0\pmod n$.
In this case, $\alpha$ has no pure strand.
\item[(ii)]
$\alpha$ is conjugate to $\epsilon^m$ for some $m\not\equiv 0\pmod{n-1}$.
In this case, $\alpha$ has only one pure strand.
\end{itemize}
\end{lemma}

\begin{corollary}\label{cor:per}
Let $\alpha$ be a periodic $n$-braid whose first strand is pure.
\begin{itemize}
\item[(i)]
If\/ $\alpha$ has at least two pure strands,
then $\alpha$ is central.
\item[(ii)]
If\/ $\alpha$ is 1-unlinked, then $\alpha$ is the identity.
\end{itemize}
\end{corollary}

\begin{proof}
(i)\ \ It is immediate from Lemma~\ref{lem:per}.

(ii)\ \ Let $\alpha$ be 1-unlinked and periodic.
Because $\alpha$ is 1-pure, it is conjugate to $\epsilon^m$.
Because $\lk(\epsilon)=1$ and $\alpha$ is 1-unlinked,
$0=\lk(\alpha)=\lk(\epsilon^m)=m\lk(\epsilon)=m$, hence
$\alpha$ is the identity.
\end{proof}

\subsection{Reducible braids}
\begin{definition}
A curve system $\C$ in $D_n$ means a finite collection
of disjoint simple closed curves in $D_n$.
It is said to be \emph{essential}\/
if each component is homotopic neither to a point
nor to a puncture nor to the boundary.
It is said to be \emph{unnested} if none of its components
encloses another component as in Figure~\ref{fig:standard}~(b).
\end{definition}

\begin{definition}
The $n$-braid group $B_n$ acts on the set of curve systems in $D_n$.
Let $\alpha*\C$ denote the left action of $\alpha\in B_n$
on the curve system $\C$ in $D_n$.
An $n$-braid $\alpha$ is said to be \emph{reducible}
if $\alpha*\C=\C$ for some essential curve system $\C$ in $D_n$.
Such a curve system $\C$ is called a \emph{reduction system} of $\alpha$.
\end{definition}

\def\temp{
If $\C$ is an essential curve system in $D_n$, then there could be punctures of $D_n$
not enclosed by any circle in $\C$.
In order to simplify the notations, define the curve system $\overline{\C}$
to contain exactly the curves of $\C$, plus one circle around each such puncture
of $D_n$.
}

\subsubsection{Canonical reduction system}
For a reduction system $\C$ of an $n$-braid $\alpha$,
let $D_\C$ be the closure of $D_n\setminus N(\C)$ in
$D_n$, where $N(\C)$ is a regular neighborhood of $\C$. The
restriction of $\alpha$ induces an automorphism on $D_\C$ that is
well defined up to isotopy. Due to Birman, Lubotzky and
McCarthy~\cite{BLM83} and Ivanov~\cite{Iva92},
for any $n$-braid $\alpha$,
there is a unique
\emph{canonical reduction system} $\R(\alpha)$ with the following
properties.
\begin{enumerate}
\item[(i)]
$\R(\alpha^m)=\R(\alpha)$ for all $m\ne 0$.

\item[(ii)]
$\R(\beta\alpha\beta^{-1})=\beta*\R(\alpha)$ for all $\beta\in B_n$.

\item[(iii)]
The restriction of $\alpha$ to each component of $D_{\R(\alpha)}$ is
either periodic or pseudo-Anosov. A reduction system with this
property is said to be \emph{adequate}.

\item[(iv)]
If $\C$ is an adequate reduction system of $\alpha$,
then $\R(\alpha)\subset\C$.
\end{enumerate}

By the properties of canonical reduction systems,
a braid $\alpha$ is reducible and non-periodic if and only if
$\R(\alpha)\ne\emptyset$.
Let $\Rext(\alpha)$ denote the collection
of the outermost components of $\R(\alpha)$.
Then $\Rext(\alpha)$ is an unnested curve system satisfying
the properties (i) and (ii).

\subsubsection{Standard reduction system}

In this paper we use a notation, introduced in~\cite{LL08},
for reducible braids with standard reduction system.

\begin{definition}
An essential curve system in $D_n$ is said to be \emph{standard}\/
if each component is isotopic to a round circle
centered at the real axis as in Figure~\ref{fig:standard}~(a).
\end{definition}

The unnested standard curve systems in $D_n$ are
in one-to-one correspondence
with the $r$-compositions of $n$ for $2\le r\le n-1$.
Recall that an ordered $r$-tuple $\n=(n_1,\ldots,n_r)$
is an \emph{$r$-composition} of $n$ if
$n_i\ge 1$ for each $i$ and $n=n_1+\cdots+n_r$.

\begin{definition}\label{def:CurSys}
For a composition $\n=(n_1,\ldots,n_r)$ of $n$,
let $\C_\n$ denote the unnested standard curve system $\cup_{n_i\ge 2}C_i$,
where each $C_i$ is a round circle, centered at the real line,
enclosing the punctures
$\{m\mid \sum_{j=1}^{i-1}n_j< m\le \sum_{j=1}^{i}n_j\}$.
For example, Figure~\ref{fig:standard}~(b)
shows the unnested standard curve system
$\C_\n$ for  $\n=(1,1,2,1,2,3)$.
\end{definition}

\begin{figure}
\begin{tabular}{ccc}
\includegraphics[scale=.6]{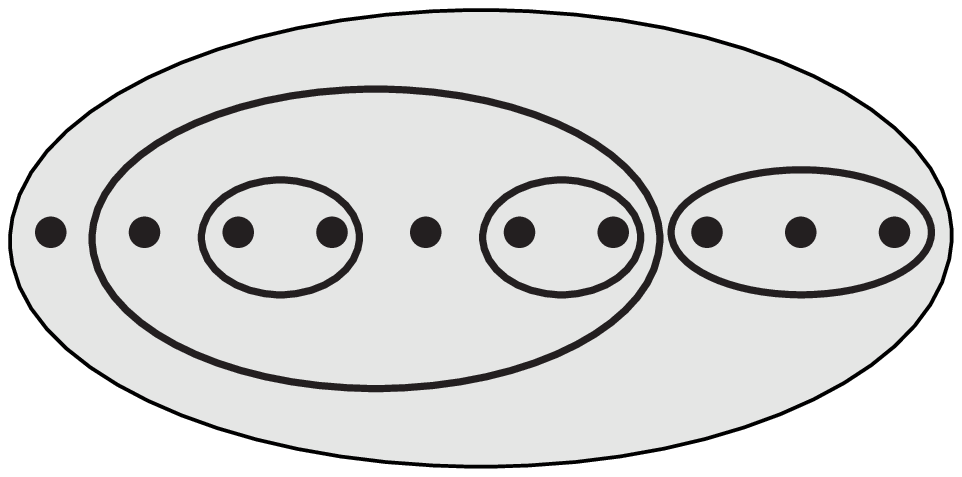}
&& \includegraphics[scale=.6]{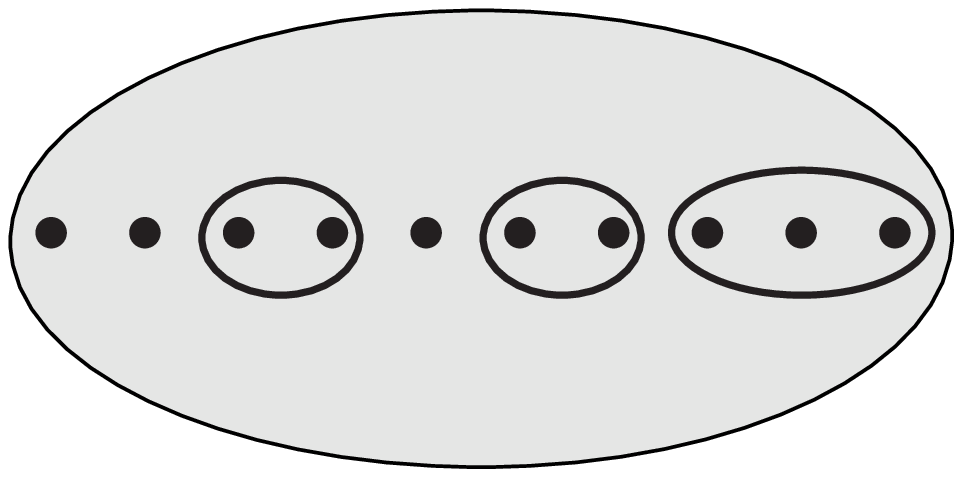}\\
(a) &\mbox{}\quad\mbox{}& (b)
\end{tabular}
\vskip -2mm
\caption{(a) shows a standard curve system in $D_{10}$.
(b) shows the unnested standard curve system
$\C_\n$ for  $\n=(1,1,2,1,2,3)$}
\label{fig:standard}
\end{figure}

The $r$-braid group $B_r$ acts on the set of $r$-compositions
of $n$ via the induced permutations: for an $r$-composition
$\n=(n_1,\cdots,n_r)$ of $n$ and
$\alpha\in B_r$ with induced permutation $\theta$,
$\alpha*\n=(n_{\theta^{-1}(1)},\ldots,n_{\theta^{-1}(r)})$.

\begin{remark}
Throughout this paper, braids and permutations
act on the left.
That is, if $\alpha$ and $\beta$ are $n$-braids,
then $(\alpha\beta)*\C = \alpha*(\beta*\C)$ for a curve
system $\C$ in $D_n$;
if $\alpha$ and $\beta$ are $r$-braids,
then $(\alpha\beta)*\n=\alpha*(\beta*\n)$
for an $r$-composition $\n$ of $n$;
if $\pi_1$ and $\pi_2$ are $n$-permutations, then
$(\pi_1\circ\pi_2)(i)=\pi_1(\pi_2(i))$ for $1\le i\le n$.
\end{remark}

\begin{definition}
Let $\n=(n_1,\cdots,n_r)$ be a composition of $n$.
\begin{itemize}
\item
Let $\alpha_0=l_1\cup\cdots\cup l_r$ be an $r$-braid
with $l_i\cap(D^2\times\{1\})=\{(i, 1)\}$ for each $i$.
We define $\myangle{\alpha_0}_\n$ as the $n$-braid obtained
from $\alpha_0$ by taking $n_i$ parallel copies of $l_i$ for each $i$.
See Figure~\ref{fig:copy}~(a).

\item
Let $\alpha_i\in B_{n_i}$ for $i=1,\ldots,r$.
We define $(\alpha_1\oplus\cdots\oplus\alpha_r)_\n$ as the $n$-braid
$\alpha_1'\alpha_2'\cdots\alpha_r'$, where each $\alpha_i'$
is the image of $\alpha_i$ under the homomorphism
$B_{n_i}\to B_n$ defined by $\sigma_j\mapsto\sigma_{n_1+\cdots+n_{i-1}+j}$.
See Figure~\ref{fig:copy}~(b).
\end{itemize}
\end{definition}

We will use the notation
$\alpha=\myangle{\alpha_0}_\n(\alpha_1\oplus\cdots\oplus\alpha_r)_\n$
throughout the paper.
See Figure~\ref{fig:copy}~(c).
The following lemma shows some elementary properties.

\begin{figure}
\begin{tabular}{ccccc}
\includegraphics[scale=1]{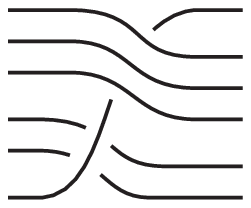} &\qquad&
\includegraphics[scale=1]{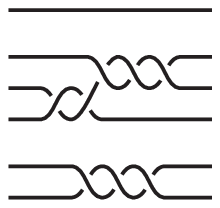} &\qquad&
\includegraphics[scale=1]{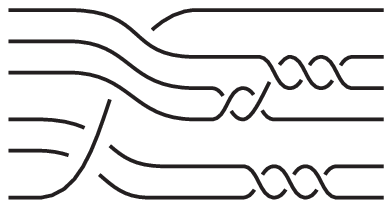} \\
\small (a) $\myangle{\sigma_1^{-1}\sigma_2}_\n$&&
\small (b) $(\sigma_1^3\oplus\sigma_1^{-2}\sigma_2^3\oplus 1)_\n$&&
\small (c) $\myangle{\sigma_1^{-1}
\sigma_2}_\n(\sigma_1^3\oplus\sigma_1^{-2}\sigma_2^3\oplus 1)_\n$
\end{tabular}
\caption{$\n=(2,3,1)$}\label{fig:copy}
\end{figure}

\begin{lemma}[{\cite[Lemmas 3.5 and 3.6]{LL08}}]\label{thm:decom}
Let\/ $\n=(n_1,\ldots,n_r)$ be a composition of\/ $n$.
\begin{enumerate}
\item[(i)]
The expression $\alpha=\myangle{\alpha_0}_\n(\alpha_1\oplus\cdots\oplus\alpha_r)_\n$
is unique, i.e. if\/
$\myangle{\alpha_0}_\n(\alpha_1\oplus\cdots\oplus\alpha_r)_\n
=\myangle{\beta_0}_\n(\beta_1\oplus\cdots\oplus\beta_r)_\n$,
then $\alpha_i=\beta_i$ for $i=0, 1,\ldots,r$.

\item[(ii)]
If $\alpha=\myangle{\alpha_0}_\n(\alpha_1\oplus\cdots\oplus\alpha_r)_\n$,
then $\alpha*\C_\n$ is standard and, further, $\alpha*\C_\n=\C_{\alpha_0\ast\n}$.
Conversely, if\/ $\alpha*\C_\n$ is standard, then $\alpha$ can be expressed as
$\alpha=\myangle{\alpha_0}_\n(\alpha_1\oplus\cdots\oplus\alpha_r)_\n$.

\item[(iii)]
$\myangle{\alpha_0}_\n(\alpha_1\oplus\cdots\oplus\alpha_r)_\n
= (\alpha_{\theta^{-1}(1)}\oplus\cdots\oplus\alpha_{\theta^{-1}(r)})_{\alpha_0\ast\n}
\myangle{\alpha_0}_\n$,
where $\theta$ is the induced permutation of\/ $\alpha_0$.

\item[(iv)]
$\myangle{\alpha_0 \beta_0}_\n
=\myangle{\alpha_0}_{\beta_0*\n}\myangle{\beta_0}_\n$.

\item[(v)]
$(\myangle{\alpha_0}_\n)^{-1}=\myangle{\alpha_0^{-1}}_{\alpha_0*\n}$.

\item[(vi)]
$
(\alpha_1\beta_1\oplus\cdots\oplus\alpha_r\beta_r)_\n
=(\alpha_1\oplus\cdots\oplus\alpha_r)_\n
(\beta_1\oplus\cdots\oplus\beta_r)_\n
$

\item[(vii)]
$(\alpha_1\oplus\cdots\oplus\alpha_r)_\n^{-1}
=(\alpha_1^{-1}\oplus\cdots\oplus\alpha_r^{-1})_\n$.

\item[(viii)]
Let $\alpha=\myangle{\alpha_0}_\n(\alpha_1\oplus\cdots\oplus\alpha_r)_\n$.
Then $\alpha = \Delta_{(n)}$ if and only if\/
$\alpha_0=\Delta_{(r)}$ and $\alpha_i=\Delta_{(n_i)}$ for $1\le i\le r$.
\end{enumerate}
\end{lemma}

\def\temp{
For an $n$-braid $\alpha$ and a composition $\n=(n_1,\ldots,n_r)$ of $n$,
if $\alpha*\C_\n$ is standard, then there exists a unique expression
$\alpha=\myangle{\alpha_0}_\n(\alpha_1\oplus\cdots\oplus\alpha_r)_\n$
by the above lemma.
We denote $\alpha_0$ by $\Ext_{\n}(\alpha)$.
For another $n$-braid $\beta$,
if $\alpha*\C_\n=\beta*\C_\n=\C_\n$, then
$\Ext_{\n}(\alpha\beta)=\Ext_{\n}(\alpha)\Ext_{\n}(\beta)$.
}

\subsection{Basic properties of $P$-pure, $P$-straight or 1-unlinked braids}

\begin{lemma}\label{lem:gamma}
Let $\alpha$, $\beta$ and $\gamma$ be $n$-braids,
and let $P$ be a subset of\/ $\{1,\ldots,n\}$.
\begin{itemize}
\item[(i)]
If\/ $\alpha$ is $P$-pure,
then $\gamma\alpha\gamma^{-1}$ is $\pi_\gamma(P)$-pure.

\item[(ii)]
If\/ $\alpha$ is $P$-straight,
then $\gamma\alpha\gamma^{-1}$ is $\pi_\gamma(P)$-straight.

\item[(iii)]
If\/ $\alpha$ is 1-unlinked and $\gamma$ is 1-pure,
then $\gamma\alpha\gamma^{-1}$ is 1-unlinked.

\item[(iv)]
If\/ both $\alpha$ and $\beta$ are
$P$-pure (resp.{} $P$-straight, 1-unlinked),
then $\alpha^p\beta^q$ is $P$-pure (resp.{} $P$-straight, 1-unlinked)
for any integers $p$ and $q$.
\end{itemize}
\end{lemma}

\begin{proof}
(i) and (ii) are obvious.

(iii)\ \ It follows from
$\lk(\gamma\alpha\gamma^{-1})=\lk(\gamma)+\lk(\alpha)-\lk(\gamma)$.

(iv)\ \ It is obvious for $P$-pureness and $P$-straightness.
The 1-unlinkedness follows from
$\lk(\alpha^p\beta^q)=p\lk(\alpha)+q\lk(\beta)=0$.
\end{proof}

\begin{lemma}\label{lem:per2}
If\/ $\alpha\in B_{n,1}$ is a periodic braid,
then there exists a 1-unlinked $n$-braid $\gamma$ such that
$\gamma\alpha\gamma^{-1}=\epsilon^m$ for some integer $m$.
\end{lemma}

\begin{proof}
If $\alpha$ is central, then we can take the identity
as the conjugating element $\gamma$.
Therefore we may assume that $\alpha$ is non-central,
hence $\alpha$ is conjugate to $\epsilon^m$
for some $m\not\equiv 0\bmod n-1$.
There exists an $n$-braid $\gamma_1$
such that $\gamma_1\alpha\gamma_1^{-1}=\epsilon^m$.
Because $\gamma_1\alpha\gamma_1^{-1}$ is $\pi_{\gamma_1}(1)$-pure
and $\epsilon^m$ has the first strand as the only pure strand,
we have $\pi_{\gamma_1}(1)=1$, that is, $\gamma_1$ is 1-pure.
Let $q=\lk(\gamma_1)$ and $\gamma=\epsilon^{-q}\gamma_1$.
Then
$$
\gamma\alpha\gamma^{-1}
=\epsilon^{-q}(\gamma_1\alpha\gamma_1^{-1})\epsilon^q
=\epsilon^{-q}\epsilon^m\epsilon^q
=\epsilon^m.
$$
Since $\gamma_1$ and $\epsilon$ are 1-pure, so is $\gamma$.
Since $\lk(\epsilon)=1$, we have
$$
\lk(\gamma)
=\lk(\gamma_1)+\lk(\epsilon^{-q})
=\lk(\gamma_1)-q=0.
$$
Therefore $\gamma$ is a conjugating element
from $\alpha$ to $\epsilon^m$,
which is 1-unlinked.
\end{proof}

\begin{figure}
\begin{tabular}{ccc}
\includegraphics[scale=.9]{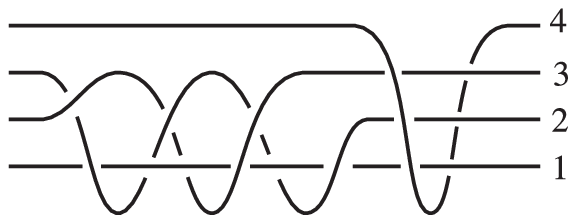} &\mbox{}\qquad\mbox{} &
\includegraphics[scale=.9]{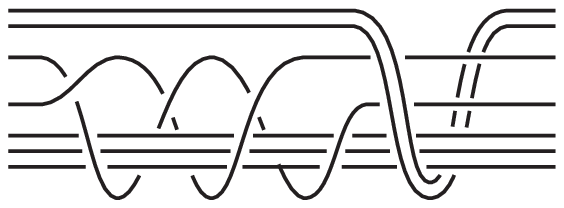} \\
(a) $\alpha=\sigma_2^{-1}\sigma_1^2\sigma_2^{-1}\sigma_1^{-2}\sigma_2^{-1}
\sigma_1^{-2}\sigma_3\sigma_2\sigma_1^2\sigma_2\sigma_3$ &&
(b) $\beta=\myangle{\alpha}_\n$ for $\n=(3,1,1,2)$
\end{tabular}
\caption{
For the above 4-braid $\alpha$, we have
$\lk_2(\alpha)=0$,
$\lk_3(\alpha)=-1$,
$\lk_4(\alpha)=1$,
hence
$\lk(\alpha)=0+(-1)+1=0$.
For the above 7-braid $\beta$, we have
$\lk_2(\beta)=\lk_3(\alpha)=0$,
$\lk_4(\beta)=0$,
$\lk_5(\beta)=-1$,
$\lk_6(\beta)=\lk_7(\alpha)=1$,
hence $\lk(\beta)=1$.}
\label{fig:lk}
\end{figure}

\begin{definition}
For a braid $\alpha\in B_{n,1}$ and an integer $2\le i\le n$,
we define the \emph{$i$-th linking number} $\lk_i(\alpha)$ of $\alpha$ as
the linking number between the first and the $i$-th strands of $\alpha$.
See Figure~\ref{fig:lk}.
\end{definition}

The following is an obvious relation between the linking number
and the $i$-th linking number.

\begin{lemma}\label{lem:LinkingNo}
Let $\alpha=\myangle{\alpha_0}_\n
(\alpha_1\oplus\alpha_2\oplus\cdots\oplus\alpha_r)_\n \in B_{n,1}$
for a composition $\n=(n_1,\ldots,n_r)$ of\/ $n$.
Then
$\lk(\alpha)=\lk(\alpha_1)
+\sum_{i=2}^r n_i\lk_i(\alpha_0)$.
\end{lemma}

\begin{definition}
For a set $P\subset\{1, 2,\ldots,n\}$ and a composition
$\n=(n_1,\ldots,n_r)$ of $n$, define the sets
$P_{\n,0}, P_{\n,1},\ldots,P_{\n,r}$ as follows:
\begin{eqnarray*}
P_{\n,i}&=&\{ 1\le j\le n_i \mid (n_1+\cdots+n_{i-1}) + j \in P \}
\qquad\mbox{for $i=1,\ldots,r$};\\
P_{\n,0} &=&\{1\le i\le r \mid P_{\n,i}\neq \emptyset\}.
\end{eqnarray*}
\end{definition}

Note that, using the above notations,
$P=\bigcup_{i=1}^r ((n_1+\cdots+n_{i-1})+P_{\n,i})$.
The following lemma is easy.

\begin{lemma}\label{lem:indu}
Let $P\subset\{1, 2, \ldots,n\}$
and $\alpha=\myangle{\alpha_0}_\n(\alpha_1\oplus\cdots\oplus\alpha_r )_\n$
for a composition $\n$ of\/ $n$.
\begin{itemize}
\item[(i)]
$\alpha$ is $P$-pure if and only if\/
$\alpha_i$'s are $P_{\n,i}$-pure for all $i=0, 1,\ldots,r$.

\item[(ii)]
$\alpha$ is $P$-straight if and only if\/
$\alpha_i$'s are $P_{\n,i}$-straight for all $i=0, 1,\ldots,r$.

\item[(iii)]
If\/ $\alpha_1$ is 1-unlinked,
then $(\alpha_1\oplus\cdots\oplus\alpha_r)_\n$ is 1-unlinked.
\def\temp{
\item[(iv)]
If\/ $\alpha_0$ is 1-unlinked and $n_2=\cdots=n_r$,
then $\myangle{\alpha_0}_\n$ is 1-unlinked.

\item[(v)]
If\/ $\alpha_0$ and $\alpha_1$ are 1-unlinked and $n_2=\cdots=n_r$,
then $\myangle{\alpha_0}_\n(\alpha_1\oplus\cdots\oplus\alpha_r)_\n$ is 1-unlinked.
}
\end{itemize}
\end{lemma}

\section{Uniqueness of roots up to conjugacy}

In this section, we prove Theorem~\ref{thm:main}.
Let us explain our strategy for proof.
Suppose we are given $P$-pure braids $\alpha$ and $\beta$
such that $\alpha^k=\beta^k$ for some nonzero integer $k$.
Note that $\alpha$ is either pseudo-Anosov, or periodic,
or reducible and non-periodic.
Lemma~\ref{lem:irred} deals with the case where
$\alpha$ is pseudo-Anosov or periodic.
Now, suppose $\alpha$ is reducible and non-periodic.
There are three cases:
$\alpha_\ext$ is pseudo-Anosov;
$\alpha_\ext$ is central;
$\alpha_\ext$ is periodic and non-central.
(Here $\alpha_\ext$ is a particular tubular braid of $\alpha$.
See Definition~\ref{def:ext-br}.)
If $\alpha_\ext$ is either pseudo-Anosov or central,
we may assume $\alpha_\ext=\beta_\ext$, and
this case is resolved in Lemma~\ref{lem:same-ext}.
For the case where $\alpha_\ext$ is periodic and non-central,
we construct a $P$-straight conjugating element from $\alpha$ to $\beta$,
and then modify this conjugating element in order to make it 1-unlinked.
Lemma~\ref{lem:per-ext} is useful in this modification.
In the end we give the proof of Theorem~\ref{thm:main}.
Due to the lemmas mentioned above, it suffices to
construct a $P$-straight conjugating element from $\alpha$ to $\beta$
for the case where $\alpha_\ext$ is periodic and non-central.

\medskip

From now on, we will say that
\emph{Theorem~\ref{thm:main} is true for $(\alpha,\beta,P,k)$}
if $(\alpha,\beta,P,k)$ is given as in Theorem~\ref{thm:main}
and there exists a $P$-straight, 1-unlinked braid $\gamma$
with $\beta=\gamma\alpha\gamma^{-1}$.

\begin{lemma}\label{lem:conj}
Let $(\alpha,\beta,P,k)$ be given as in Theorem~\ref{thm:main}.
\begin{itemize}
\item[(i)]
Let $\chi$ be a 1-pure $n$-braid.
If Theorem~\ref{thm:main} is true for
$(\chi\alpha\chi^{-1}, \chi\beta\chi^{-1},\pi_\chi(P),k)$,
then it is also true for $(\alpha,\beta,P,k)$.

\item[(ii)]
Let $\chi$ be a $P$-straight, 1-unlinked $n$-braid
with $\chi\beta^k=\beta^k\chi$.
If Theorem~\ref{thm:main} is true for
$(\alpha, \chi\beta\chi^{-1},P,k)$,
then it is also true for $(\alpha,\beta,P,k)$.
\end{itemize}
\end{lemma}

\begin{proof}
(i)\ \
Note that $(\chi\alpha\chi^{-1})^k=(\chi\beta\chi^{-1})^k$,
that both $\chi\alpha\chi^{-1}$ and $\chi\beta\chi^{-1}$
are $\pi_\chi(P)$-pure by Lemma~\ref{lem:gamma}~(i),
and that $1\in\pi_\chi(P)$ because $1\in P$ and $\chi$ is 1-pure.
Suppose Theorem~\ref{thm:main} is true for
$(\chi\alpha\chi^{-1}, \chi\beta\chi^{-1},\pi_\chi(P),k)$, that is,
there exists a $\pi_\chi(P)$-straight, 1-unlinked $n$-braid $\gamma_1$
with $\chi\beta\chi^{-1}=\gamma_1(\chi\alpha\chi^{-1})\gamma_1^{-1}$.
Let $\gamma=\chi^{-1}\gamma_1\chi$, then $\beta=\gamma\alpha\gamma^{-1}$.
Since $\gamma_1$ is $\pi_\chi(P)$-straight, $\gamma$ is $P$-straight
by Lemma~\ref{lem:gamma}~(ii).
Since $\gamma_1$ is 1-unlinked and $\chi$ is 1-pure,
$\gamma$ is 1-unlinked by Lemma~\ref{lem:gamma}~(iii).

\smallskip(ii)\ \
Since $\chi$ commutes with $\beta^k$,
$(\chi\beta\chi^{-1})^k=\chi\beta^k\chi^{-1}=\beta^k=\alpha^k$.
Because both $\beta$ and $\chi$  are $P$-pure,
$\chi\beta\chi^{-1}$ is $P$-pure by Lemma~\ref{lem:gamma}~(iv).
Suppose Theorem~\ref{thm:main} is true for $(\alpha, \chi\beta\chi^{-1},P,k)$,
that is, there exists a $P$-straight, 1-unlinked $n$-braid $\gamma_1$
such that $\chi\beta\chi^{-1}=\gamma_1\alpha\gamma_1^{-1}$.
Let $\gamma=\chi^{-1}\gamma_1$, then $\beta=\gamma\alpha\gamma^{-1}$.
Since both $\gamma_1$ and $\chi$ are $P$-straight and 1-unlinked,
$\gamma$ is $P$-straight and 1-unlinked by Lemma~\ref{lem:gamma}~(iv).
\end{proof}

\begin{lemma}\label{lem:irred}
Let $(\alpha,\beta,P,k)$ be given as in Theorem~\ref{thm:main}.
If\/ $\alpha$ is either pseudo-Anosov or periodic,
then Theorem~\ref{thm:main} is true for $(\alpha,\beta,P,k)$.
\end{lemma}

\begin{proof}
If $\alpha$ is pseudo-Anosov, then $\alpha=\beta$ by Lemma~\ref{lem:NT}.
If $\alpha$ is central, then $\alpha=\beta$
because $\beta$ is conjugate to $\alpha$.
In these two cases, we can take the identity as the desired
conjugating element $\gamma$.

Suppose that $\alpha$ is periodic and non-central.
Then both $\alpha$ and $\beta$ are conjugate to $\epsilon^m$
for some $m\not\equiv 0\bmod n-1$ by Lemma~\ref{lem:per}
since they are 1-pure and non-central.
By Lemma~\ref{lem:per2}, there exist 1-unlinked $n$-braids
$\gamma_1$ and $\gamma_2$ such that
$\gamma_1\alpha\gamma_1^{-1}=\epsilon^m=\gamma_2\beta\gamma_2^{-1}$.
Let $\gamma=\gamma_2^{-1}\gamma_1$,
then $\beta=\gamma\alpha\gamma^{-1}$.
Because both $\gamma_1$ and $\gamma_2$ are 1-unlinked, so is $\gamma$.
Because the first strand is the only pure strand of $\alpha$
and $1\in P$, we have $P=\{1\}$.
Therefore $\gamma$ is $P$-straight.
\end{proof}

\begin{definition}\label{def:ext-br}
Let $\alpha$ be an $n$-braid with $\Rext(\alpha)$ standard.
Then there exists a composition $\n=(n_1,\ldots, n_r)$ of $n$ such that
$\Rext(\alpha)=\C_\n$ and $\alpha$ can be expressed as
$$\alpha = \myangle{\alpha_0}_\n
(\alpha_1\oplus\cdots\oplus\alpha_r)_\n.$$
In this case, the tubular $r$-braid $\alpha_0$ of $\alpha$ is specially denoted
by $\alpha_\ext$.
\end{definition}

Note that for non-periodic reducible braids $\alpha$ and $\beta$,
if $\alpha^k =\beta^k$ for a nonzero integer $k$, then
$$\emptyset\neq\Rext(\alpha)=\Rext(\alpha^k)=\Rext(\beta^k)=\Rext(\beta).$$

\begin{lemma}\label{lem:same-ext}
Let $(\alpha,\beta,P,k)$ be given as in Theorem~\ref{thm:main}.
If\/ $\Rext(\alpha)$ is standard and $\alpha_\ext=\beta_\ext$,
then Theorem~\ref{thm:main} is true for $(\alpha,\beta,P,k)$.
\end{lemma}

\begin{proof}
We will show this lemma by induction on the braid index $n$.
If $n=2$, Theorem~\ref{thm:main} is obvious because $B_{2,1}$ is
infinite cyclic generated by $\sigma_1^2$:
if $\alpha=\sigma_1^{2p}$ and $\beta=\sigma_1^{2q}$,
then $\alpha^k=\beta^k$ implies $p=q$, and hence
$\alpha=\beta$ and the identity is a conjugating element from
$\alpha$ to $\beta$.

Suppose that $n>2$ and that the theorem is true
for braids with less than $n$ strands.
Let $\Rext(\alpha)=\C_\n$ for an $r$-composition $\n$ of $n$.
Let $\alpha_0=\alpha_\ext\in B_r$.
Since $\alpha$ is 1-pure, so is $\alpha_0$, that is, $\pi_{\alpha_0}(1)=1$.
Let $\{ z_2, z_3, \ldots, z_m\}$ be the set of all points other than 1
each of which is fixed by $\pi_{\alpha_0}$.

\begin{claim}
Without loss of generality,
we may assume that $\{ z_2, \ldots, z_m\} =\{2, \ldots, m\}$
(i.e. $\pi_{\alpha_0}(i)=i$ for all $1\le i\le m$)
and each of the other cycles of $\pi_{\alpha_0}$ is of the form
$(i+r_i,\ldots, i+2, i+1)$ for some $i\ge m$ and $r_i\ge 2$.
\end{claim}

\begin{proof}[Proof of Claim]
Choose an $r$-permutation $\theta$ such that $\theta(1)=1$,
$\theta(\{ z_2, \ldots, z_m\})=\{ 2, \ldots, m\}$
and each cycle (of length $\ge 2$)
of $\theta\pi_{\alpha_0}\theta^{-1}$ is of the form $(i+r_i,\ldots, i+2, i+1)$.
Note that $\theta\pi_{\alpha_0}\theta^{-1}$ fixes each point of $\{1,\ldots,m\}$.
Let $\zeta_0$ be an $r$-braid whose induced permutation is $\theta$,
and let $\zeta=\myangle{\zeta_0}_\n$.
Since $\zeta_0$ is 1-pure, $\zeta$ is also 1-pure.
Applying Lemma~\ref{lem:conj}~(i) to $\zeta$ and $(\alpha,\beta,P, k)$,
it suffices to show that Theorem~\ref{thm:main} is true for
$(\zeta\alpha\zeta^{-1}, \zeta\beta\zeta^{-1}, \pi_{\zeta}(P), k)$.
Note that
$\Rext(\zeta\alpha\zeta^{-1})=\zeta*\Rext(\alpha)=\C_{\zeta_0*\n}$
is standard and that
$(\zeta\alpha\zeta^{-1})_\ext
=\zeta_0\alpha_\ext\zeta_0^{-1}
=\zeta_0\beta_\ext\zeta_0^{-1}
=(\zeta\beta\zeta^{-1})_\ext$.
\end{proof}

Using the above claim,
we assume that $\pi_{\alpha_0}(i)=i$ for all $1\le i\le m$
and each of the other cycles of $\pi_{\alpha_0}$ is of the form
$(i+r_i,\ldots, i+2, i+1)$ for some $i\ge m$ and $r_i\ge 2$.
Then $\n$, $\alpha$ and $\beta$ are as follows:
\begin{eqnarray*}
\n &=& ( n_1, \ldots, n_m, \underbrace{n_{m+1},\ldots, n_{m+1}}_{r_{m+1}},
\ldots, \underbrace{n_s,\ldots, n_s}_{r_s}),\\
\alpha & = & \myangle{\alpha_0}_{\n}
  (\alpha_1\oplus\cdots\oplus\alpha_m\oplus(\alpha_{m+1,1}\oplus\cdots\oplus\alpha_{m+1,r_{m+1}})
  \oplus\cdots\oplus(\alpha_{s,1}\oplus\cdots\oplus\alpha_{s,r_s}))_\n, \\
\beta & = & \myangle{\alpha_0}_{\n}
  (\beta_1\oplus\cdots\oplus\beta_m\oplus(\beta_{m+1,1}\oplus\cdots\oplus\beta_{m+1,r_{m+1}})
  \oplus\cdots\oplus(\beta_{s,1}\oplus\cdots\oplus\beta_{s,r_s}))_\n.
\end{eqnarray*}

\medskip

By raising the power $k$ large enough, we may assume that
the lengths $r_i$ of the cycles of $\pi_{\alpha_0}$ are all divisors of $k$.
Let $k=r_ip_i$ for $m< i\le s$.
Then
\begin{eqnarray*}
\alpha^k&=& \myangle{\alpha_0^k}_{\n}
  (\alpha_1^k \oplus\cdots\oplus \alpha_m^k
  \oplus(\tilde\alpha_{m+1,1}^{p_{m+1}}\oplus\cdots\oplus\tilde\alpha_{m+1,r_{m+1}}^{p_{m+1}})
  \oplus\cdots\oplus
  (\tilde\alpha_{s,1}^{p_s}\oplus\cdots\oplus\tilde\alpha_{s,r_s}^{p_s}))_\n,\\
\beta^k&=& \myangle{\alpha_0^k}_{\n}
  (\beta_1^k \oplus\cdots\oplus \beta_m^k
  \oplus(\tilde\beta_{m+1,1}^{p_{m+1}}\oplus\cdots\oplus\tilde\beta_{m+1,r_{m+1}}^{p_{m+1}})
  \oplus\cdots\oplus
  (\tilde\beta_{s,1}^{p_s}\oplus\cdots\oplus\tilde\beta_{s,r_s}^{p_s}))_\n,
\end{eqnarray*}
where
\begin{eqnarray*}
\tilde\alpha_{i,j} &=&
  \alpha_{i,j-r_i+1}\alpha_{i,j-r_i+2}\cdots\alpha_{i,j-1}\alpha_{i,j},\\
\tilde\beta_{i,j} &=&
  \beta_{i,j-r_i+1}\beta_{i,j-r_i+2}\cdots\beta_{i,j-1}\beta_{i,j}
\end{eqnarray*}
for $m< i\le s$ and $1\le j\le r_i$.
Hereafter we regard the second index $j$ of $(i,j)$ as being taken modulo $r_i$.
Since $\alpha^k=\beta^k$, one has
\begin{eqnarray*}
& \alpha_i^k =\beta_i^k & \quad\mbox{for $1\le i\le m$}, \\
& \tilde\alpha_{i,j}^{p_i}=\tilde\beta_{i,j}^{p_i}
& \quad\mbox{for $m < i\le s$ and $1\le j\le r_i$}.
\end{eqnarray*}

Recall that $\alpha$ and $\beta$ are $P$-pure, hence
$\alpha_i$ and $\beta_i$ are $P_{\n,i}$-pure for $1\le i\le m$
by Lemma~\ref{lem:indu}.
Recall also that the induced permutation of $\alpha_0$ fixes no point $i>m$,
hence $P_{\n,i}=\emptyset$ for $i>m$.

\medskip
From now on, we will construct an $n$-braid $\gamma$
such that $\beta=\gamma\alpha\gamma^{-1}$.
It will be of the form
$$
\gamma = (\gamma_1\oplus\cdots\oplus\gamma_m\oplus
  (\gamma_{m+1,1}\oplus\cdots\oplus\gamma_{m+1,r_{m+1}})
  \oplus\cdots\oplus(\gamma_{s,1}\oplus\cdots\oplus\gamma_{s,r_s}))_\n,
$$
where $\gamma_1$ is 1-unlinked and $P_{\n,1}$-straight,
and $\gamma_i$ is $P_{\n,i}$-straight for $2\le i\le m$.
Then $\gamma$ is 1-unlinked by Lemma~\ref{lem:indu}~(iii)
because $\gamma_1$ is 1-unlinked.
And $\gamma$ is $P$-straight by Lemma~\ref{lem:indu}~(ii)
because $\gamma_i$ is $P_{\n,i}$-straight for $1\le i\le m$
and $P_{\n,i}=\emptyset$ for $i>m$.

\smallskip
Note that $\alpha_1^k=\beta_1^k$ and
that $1\in P_{\n,1}$ because $1\in P$.
By the induction hypothesis on the braid index,
there exists a $P_{\n,1}$-straight, 1-unlinked $n_1$-braid $\gamma_1$
with $\beta_1=\gamma_1\alpha_1\gamma_1^{-1}$.

\smallskip
Let $2\le i\le m$.
Note that $\alpha_i^k=\beta_i^k$.
If $P_{\n,i} =\emptyset$, there is an $n_i$-braid $\gamma_i$
such that $\beta_i = \gamma_i\alpha_i\gamma_i^{-1}$ by~\cite{Gon03}.
Suppose $P_{\n,i} \neq\emptyset$.
Then there is an $n_i$-braid $\zeta_i$ with $1\in\pi_{\zeta_i}(P_{\n,i})$.
Since $\alpha_i$ and $\beta_i$ are $P_{\n,i}$-pure,
$\zeta_i\alpha_i\zeta_i^{-1}$ and $\zeta_i\beta_i\zeta_i^{-1}$
are $\pi_{\zeta_i}(P_{\n,i})$-pure $n_i$-braids
with $(\zeta_i\alpha_i\zeta_i^{-1})^k = (\zeta_i\beta_i\zeta_i^{-1})^k$.
By the induction hypothesis on the braid index,
Theorem~\ref{thm:main} is true for
$(\zeta_i\alpha_i\zeta_i^{-1}, \zeta_i\beta_i\zeta_i^{-1},
\pi_{\zeta_i}(P_{\n,i}), k)$,
hence there exists a $\pi_{\zeta_i}(P_{\n,i})$-straight
$n_i$-braid $\chi_i$
such that $\zeta_i\beta_i\zeta_i^{-1}
= \chi_i (\zeta_i\alpha_i\zeta_i^{-1}) \chi_i^{-1}$.
Let $\gamma_i = \zeta_i^{-1}\chi_i\zeta_i$.
Then $\gamma_i$ is a $P_{\n,i}$-straight $n_i$-braid with
$\beta_i = \gamma_i\alpha_i\gamma_i^{-1}$.

\smallskip
Recall that $\tilde\alpha_{i,r_i}^{p_i}
=\tilde\beta_{i,r_i}^{p_i}$ for all $m< i\le s$,
which implies that there are $\zeta_i\in B_{n_i}$ with
$$
\tilde\beta_{i,r_i}=\zeta_i\tilde\alpha_{i,r_i}\zeta_i^{-1}.
$$
For $m < i\le s$ and $1\le j\le r_i$, define $\gamma_{i,j}$ by
$$
\gamma_{i,j} = (\beta_{i,j}^{-1}\cdots\beta_{i,2}^{-1}\beta_{i,1}^{-1})\zeta_i
(\alpha_{i,1}\alpha_{i,2}\cdots\alpha_{i,j}).
$$
Then, for $m< i\le s$ and $1< j \le r_i$,
\begin{eqnarray*}
\gamma_{i,r_i}\alpha_{i,1}\gamma_{i,1}^{-1}
&=&  (\beta_{i,r_i}^{-1}\cdots\beta_{i,1}^{-1}\zeta_i
  \alpha_{i,1}\cdots\alpha_{i,r_i})
  \alpha_{i,1} (\alpha_{i,1}^{-1}\zeta_i^{-1}\beta_{i,1})
  = \tilde\beta_{i,r_i}^{-1}\zeta_i\tilde\alpha_{i,r_i}\zeta_i^{-1}\beta_{i,1}
  = \beta_{i,1}, \\
\gamma_{i,j-1}\alpha_{i,j}\gamma_{i,j}^{-1}
&=& (\beta_{i,j-1}^{-1}\cdots\beta_{i,1}^{-1}\zeta_i
  \alpha_{i,1}\cdots\alpha_{i,j-1})
  \alpha_{i,j}
  (\alpha_{i,j}^{-1}\cdots\alpha_{i,1}^{-1}\zeta_i^{-1}
  \beta_{i,1}\cdots\beta_{i,j})
  = \beta_{i,j}.
\end{eqnarray*}
Therefore
$$
\gamma_{i,j-1}\alpha_{i,j}\gamma_{i,j}^{-1}
=\beta_{i,j}
\qquad\mbox{for $m< i\le s$ and $1\le j \le r_i$}.
$$

\smallskip
So far, we have constructed the desired $P$-straight and 1-unlinked $n$-braid
$$
\gamma = (\gamma_1\oplus\cdots\oplus\gamma_m\oplus
  (\gamma_{m+1,1}\oplus\cdots\oplus\gamma_{m+1,r_{m+1}})
  \oplus\cdots\oplus(\gamma_{s,1}\oplus\cdots\oplus\gamma_{s,r_s}))_\n.
$$
It remains to show $\beta=\gamma\alpha\gamma^{-1}$,
which will be done by a direct computation.
In the following, $\bigoplus_{i=1}^\ell \chi_i$ means
$\chi_1\oplus \chi_2\oplus\cdots\oplus \chi_\ell$.
\begin{eqnarray*}
\gamma\alpha
&=& (\gamma_1\oplus\cdots\oplus\gamma_m \oplus
      \bigoplus_{i=m+1}^s \bigoplus_{j=1}^{r_i}\gamma_{i,j})_\n
  \cdot\myangle{\alpha_0}_{\n}
      (\alpha_1\oplus\cdots\oplus\alpha_m\oplus
      \bigoplus_{i=m+1}^s \bigoplus_{j=1}^{r_i}\alpha_{i,j})_\n \\
&=& \myangle{\alpha_0}_{\n}\cdot
  (\gamma_1\oplus\cdots\oplus\gamma_m \oplus
      \bigoplus_{i=m+1}^s \bigoplus_{j=1}^{r_i}\gamma_{i,j-1})_\n
  \cdot
      (\alpha_1\oplus\cdots\oplus\alpha_m\oplus
      \bigoplus_{i=m+1}^s \bigoplus_{j=1}^{r_i}\alpha_{i,j})_\n \\
& = & \myangle{\alpha_0}_{\n}
  (\gamma_1\alpha_1 \oplus\cdots\oplus\gamma_m\alpha_m
  \oplus \bigoplus_{i=m+1}^s\bigoplus_{j=1}^{r_i}
    \gamma_{i,j-1}\alpha_{i,j})_\n, \\
\beta\gamma
&=&
  \myangle{\alpha_0}_{\n}
      (\beta_1\oplus\cdots\oplus\beta_m\oplus
      \bigoplus_{i=m+1}^s \bigoplus_{j=1}^{r_i}\beta_{i,j})_\n
  \cdot (\gamma_1\oplus\cdots\oplus\gamma_m \oplus
      \bigoplus_{i=m+1}^s \bigoplus_{j=1}^{r_i}\gamma_{i,j})_\n\\
&=& \myangle{\alpha_0}_{\n}
  (\beta_1\gamma_1\oplus\cdots\oplus\beta_m\gamma_m \oplus
      \bigoplus_{i=m+1}^s \bigoplus_{j=1}^{r_i}\beta_{i,j}\gamma_{i,j})_\n.
\end{eqnarray*}
Because
$\gamma_i\alpha_i\gamma_i^{-1}=\beta_i$ for $1\le i\le m$
and $\gamma_{i,j-1}\alpha_{i,j}\gamma_{i,j}^{-1}=\beta_{i,j}$
for $m< i\le s$ and $1\le j \le r_i$,
we have $\gamma\alpha=\beta\gamma$, and hence $\gamma\alpha\gamma^{-1}=\beta$.
\end{proof}

\begin{definition}
Let $r$, $s$ and $d$ be integers with $s\ge 2$, $d\ge 1$ and $r=ds+1$.
For $1\le j\le d$, define an $r$-braid $\mu_{s,j}$ as
$$
\mu_{s,j}=(\sigma_{js}\sigma_{js-1}\cdots\sigma_2\sigma_1)
(\sigma_1\sigma_2\cdots\sigma_{(j-1)s}\sigma_{(j-1)s+1}).
$$
Define $\mu_s$ as $\mu_s =\mu_{s,1}\mu_{s,2}\cdots\mu_{s,d}$.
See Figure~\ref{fig:mu} for the case $r=7$, $s=3$ and $d=2$.
\end{definition}

\begin{figure}
\tabcolsep=1em
\begin{tabular}{ccc}
\includegraphics[scale=.8]{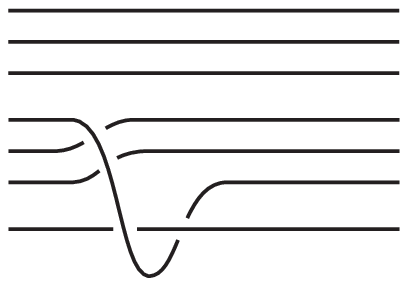} &
\includegraphics[scale=.8]{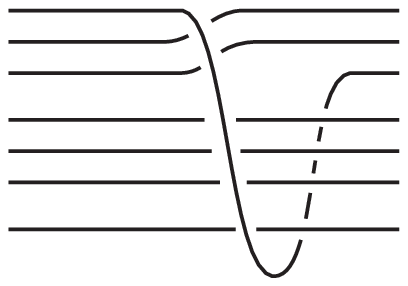} &
\includegraphics[scale=.8]{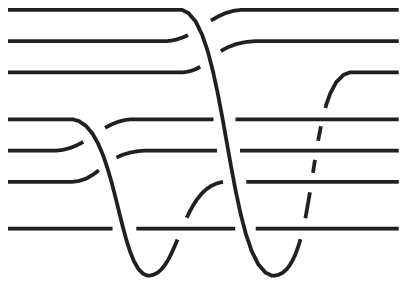} \\
(a) $\mu_{3,1}$ &
(b) $\mu_{3,2}$ &
(c) $\mu_{3}=\mu_{3,1}\mu_{3,2}$
\end{tabular}
\caption{$\mu_{3,1}$, $\mu_{3,2}$ and $\mu_3$ when $r=7$}
\label{fig:mu}
\end{figure}

It is easy to see the following.
\begin{itemize}
\item $\mu_s$ is conjugate to $\epsilon^d_{(r)}$
because $(\mu_s)^s=\Delta_{(r)}^2=(\epsilon^d_{(r)})^s$.
\item For any $1\le i, j\le d$,
$\mu_{s,i}$ and $\mu_{s,j}$ mutually commute.
\item $\lk(\mu_{s,j})=1$ for $1\le j\le d$.
\end{itemize}

\begin{lemma}\label{lem:per-ext}
Let $P$ be a subset of\/ $\{1,\ldots,n\}$ with $1\in P$.
Let $\alpha$ be a $P$-pure $n$-braid
with $\Rext(\alpha)$ standard, hence
$\Rext(\alpha)=\C_\n$ for a composition $\n=(n_1,\ldots,n_r)$ of\/ $n$.
Let $\alpha_\ext$ be periodic and non-central.
\begin{itemize}
\item[(i)]
For each $2\le i\le r$, there exists
a $P$-straight $n$-braid $\gamma$ such that
$\gamma\alpha=\alpha\gamma$ and $\lk(\gamma)=n_i$.

\item[(ii)]
Let $\chi = \myangle{\chi_0 }_{\n} ( \chi_1\oplus\cdots\oplus\chi_r)_\n$
be $P$-straight such that $\chi_1$ is 1-unlinked.
Then there exists a $P$-straight $n$-braid $\gamma$ such that
$\gamma\alpha =\alpha\gamma$ and $\lk(\gamma)= -\lk(\chi)$.
\end{itemize}
\end{lemma}

\begin{proof}
(i) \ \
Note that $\alpha_\ext$ is 1-pure because $\alpha$ is 1-pure.
In addition, $\alpha_\ext$ is periodic and non-central.
Thus $\alpha_\ext$ is conjugate to $\epsilon_{(r)}^m$ for some $m\not\equiv 0 \bmod r-1$.
Let $d=\gcd(m,r-1)$, $m=dt$ and $r-1=ds$.

\begin{claim}
Without loss of generality, we may assume $\alpha_\ext=\mu_s^t$.
\end{claim}

\begin{proof}[Proof of Claim]
Assume that (i) holds for braids $\alpha'$ with $\alpha'_\ext=\mu_s^t$.
Since $\alpha_\ext$ is conjugate to $\epsilon_{(r)}^{m} =\epsilon_{(r)}^{dt}$
and $\epsilon_{(r)}^{d}$ is conjugate to $\mu_s$,
$\mu_s^t$ is conjugate to $\alpha_\ext$.
Since both $\alpha_\ext$ and $\mu_{s}^{t}$ are 1-pure braids
that are periodic and non-central,
they have the first strand as the only pure strand.
Thus there exists a 1-pure $r$-braid $\zeta_0$ such that
$\mu_{s}^{t} = \zeta_0\alpha_\ext\zeta_0^{-1}$.
Let
$$
\zeta=\myangle{\zeta_0}_\n\quad\mbox{and}\quad
\beta = \zeta\alpha\zeta^{-1}.
$$
Since $\alpha$ is $P$-pure, $\beta$ is $\pi_{\zeta}(P)$-pure.
Since $\zeta$ is 1-pure and $1\in P$, we have $1\in\pi_\zeta(P)$.
Since $\Rext(\beta)=\zeta\ast\Rext(\alpha)=\zeta*\C_{\n}=\C_{\zeta_0\ast\n}$,
$\Rext(\beta)$ is standard and
$\beta_\ext=\zeta_0\alpha_\ext\zeta_0^{-1}=\mu_{s}^{t}$.
Fix any $2\le i\le r$.
Since $\zeta_0\ast\n=(n_1, n'_2,\ldots,n'_r)$, where
$(n'_2,\ldots,n'_r)$ is a rearrangement of $(n_2,\ldots,n_r)$,
there exists $2\le j\le r$ such that $n_i=n'_j$.
By the assumption, there exists a $\pi_{\zeta}(P)$-straight
$n$-braid $\chi$ such that $\chi\beta\chi^{-1}=\beta$ and $\lk(\chi)=n'_j$.
Let $\gamma=\zeta^{-1}\chi\zeta$.
Then $\gamma\alpha\gamma^{-1}=\alpha$.
Because $\gamma$ is $P$-straight
with $\lk(\gamma)=\lk(\chi)=n'_j=n_i$,
we are done.
\end{proof}

Now, we assume $\alpha_\ext=\mu_s^t$.
Then $\alpha$ can be expressed as
$$
\alpha=\myangle{\mu_s^t}_\n(\alpha_1\oplus
(\alpha_{1,1}\oplus\alpha_{1,2}\oplus\cdots\oplus\alpha_{1,s})\oplus\cdots\oplus
(\alpha_{d,1}\oplus\alpha_{d,2}\oplus\cdots\oplus\alpha_{d,s}) )_\n.
$$
For convenience, let $[k,\ell]$ denote the integer $(k-1)s+\ell+1$
for $1\le k\le d$ and $1\le \ell\le s$. Then
$$
\n=(n_1,n_2,\ldots,n_r)
=(n_1,\underbrace{n_{[1,1]},n_{[1,2]},\ldots,n_{[1,s]}}_s,
\ldots,\underbrace{n_{[d,1]},n_{[d,2]},\ldots,n_{[d,s]}}_s).
$$
Hereafter we regard the second index $\ell$ of $[k,\ell]$
as being taken modulo $s$.
Notice the following.
\begin{itemize}
\item
The induced permutation of $\mu_s^t$
fixes 1 and maps $[k,\ell]$ to $[k,\ell-t]$.
Because $\gcd(s, t)=1$, the induced permutation of $\mu_s^t$ has
a single fixed point and each of the other cycles has length $s$.
Therefore
$$P=P_{\n,1} \quad\mbox{and}\quad
n_{[k,1]} = n_{[k,2]} =\cdots = n_{[k,s]} \quad\mbox{for $1\le k\le d$}.
$$

\item
The induced permutation of $\mu_{s,j}$
maps $[j,\ell]$ to $[j,\ell -1]$ for $1\le \ell\le s$,
and it fixes the other points.
\item
$\lk_{[j,1]}(\mu_{s,j})=1$, and
$\lk_{[k,\ell]}(\mu_{s,j})=0$ if $(k,\ell)\ne(j,1)$.
\end{itemize}

Because $\gcd(s, t)=1$,
there exist integers $a > 0$ and $b$ such that $a t+ b s=1$.
Fix any $2\le i\le r$.
Then $n_i$ is equal to $n_{[j,1]}$ for some $1\le j\le d$.
Define an $n$-braid $\gamma$ to be
$$
\gamma=\myangle{\mu_{s,j}}_\n(\gamma_1\oplus
(\gamma_{1,1}\oplus\gamma_{1,2}\oplus\cdots\oplus\gamma_{1,s})\oplus\cdots\oplus
(\gamma_{d,1}\oplus\gamma_{d,2}\oplus\cdots\oplus\gamma_{d,s}) )_\n,
$$
where
$\gamma_1=1$, $\gamma_{k, \ell}=1$ for $k\ne j$
and
$$
\gamma_{j,\ell}
= \alpha_{j, \ell-(a-1)t}\alpha_{j, \ell-(a-2)t}\cdots
\alpha_{j, \ell-2t} \alpha_{j, \ell-t} \alpha_{j, \ell}
\quad\mbox{for $1\le \ell\le s$}.
$$
Then $\gamma$ is $P$-straight because $P = P_{\n,1}$,
$\mu_{s,j}$ is 1-pure, and $\gamma_1=1$.
In addition, by Lemma~\ref{lem:LinkingNo},
$$
\lk(\gamma)
=\lk(\gamma_1)+\sum_{k=1}^d\sum_{\ell=1}^s
  n_{[k,\ell]} \lk_{[k,\ell]}(\mu_{s,j})
=n_{[j,1]}\lk_{[j,1]}(\mu_{s,j})=n_{[j,1]}=n_i.
$$

\smallskip
Now, it remains to show $\alpha\gamma = \gamma\alpha$.
We will do it by a straightforward computation
together with the following claim.

\begin{claim}
For $1\le \ell\le s$, we have
$\alpha_{j,\ell -1}\gamma_{j,\ell} = \gamma_{j,\ell-t} \alpha_{j,\ell}$.
\end{claim}

\begin{proof}[Proof of Claim]
Recall that
$\gamma_{j,\ell} = \alpha_{j, \ell-(a-1)t}\alpha_{j, \ell-(a-2)t}\cdots
\alpha_{j, \ell-2t} \alpha_{j, \ell-t} \alpha_{j, \ell}$.
Hence
\begin{eqnarray*}
\alpha_{j,\ell -1}\gamma_{j,\ell} &=& \alpha_{j,\ell -1}
\alpha_{j, \ell-(a-1)t}\alpha_{j, \ell-(a-2)t}\cdots
\alpha_{j, \ell-2t} \alpha_{j, \ell-t} \alpha_{j, \ell}, \\
\gamma_{j,\ell-t} \alpha_{j,\ell} &=&
\alpha_{j, \ell-at}\alpha_{j, \ell-(a-1)t}\alpha_{j, \ell-(a-2)t}\cdots
\alpha_{j, \ell-2t} \alpha_{j, \ell-t}\alpha_{j,\ell}.
\end{eqnarray*}
Notice that $\alpha_{j,\ell -1} = \alpha_{j, \ell-at}$ because $at\equiv 1 \bmod s$.
Therefore $\alpha_{j,\ell -1}\gamma_{j,\ell} = \gamma_{j,\ell-t} \alpha_{j,\ell}$.
\end{proof}

For simplicity of notations, let
$$
\begin{array}{ll}
\tilde\alpha_k=(\alpha_{k,1}\oplus\cdots\oplus\alpha_{k,s})_{\n_k},\qquad
&\tilde\alpha_k^{(p)}=(\alpha_{k,p+1}\oplus\cdots\oplus\alpha_{k,p+s})_{\n_k},\\
\tilde\gamma_k=(\gamma_{k,1}\oplus\cdots\oplus\gamma_{k,s})_{\n_k},\qquad
&\tilde\gamma_k^{(p)}=(\gamma_{k,p+1}\oplus\cdots\oplus\gamma_{k,p+s})_{\n_k},
\end{array}
$$
where $1\le k\le d$, $\n_k=(n_{[k,1]},\ldots,n_{[k,s]})$
and $p$ is an integer.
Then
$$
\alpha=\myangle{\mu_s^t}_\n(\alpha_1\oplus
\tilde\alpha_1\oplus\cdots\oplus\tilde\alpha_d)_\n
\quad\mbox{and}\quad
\gamma=\myangle{\mu_{s,j}}_\n(\gamma_1\oplus
\tilde\gamma_1\oplus\cdots\oplus\tilde\gamma_d)_\n.
$$
Because $\gamma_1=1$ and $\tilde\gamma_k=1$ for $k\ne j$,
we just write $\gamma=\myangle{\mu_{s,j}}_\n(\cdots\oplus 1
\oplus\tilde\gamma_j\oplus 1\oplus\cdots)_\n$.
Then
\begin{eqnarray*}
\alpha\gamma &=&
\myangle{\mu_s^t}_\n(\alpha_1\oplus\tilde\alpha_1
  \oplus\cdots\oplus\tilde\alpha_d)_\n
  \cdot
\myangle{\mu_{s,j}}_\n(\cdots\oplus 1
  \oplus\tilde\gamma_j\oplus 1\oplus\cdots)_\n\\
&=& \myangle{\mu_s^t}_\n\cdot \myangle{\mu_{s,j}}_\n\cdot
  (\alpha_1\oplus \cdots\oplus\tilde\alpha_{j-1}
  \oplus\tilde\alpha_j^{(-1)}\oplus\tilde\alpha_{j+1}
  \oplus\cdots\oplus\tilde\alpha_d)_\n
  \cdot(\cdots\oplus1\oplus\tilde\gamma_j\oplus1\oplus\cdots)_\n \\
&=& \myangle{\mu_s^t\mu_{s,j}}_\n
  (\alpha_1\oplus \tilde\alpha_1\oplus\cdots\oplus\tilde\alpha_{j-1}
  \oplus\tilde\alpha_j^{(-1)}\tilde\gamma_j\oplus\tilde\alpha_{j+1}
  \oplus\cdots\oplus\tilde\alpha_d)_\n,\\
\gamma\alpha &=&
\myangle{\mu_{s,j}}_\n(\cdots\oplus 1\oplus\tilde\gamma_j \oplus 1
  \oplus\cdots)_\n
  \cdot
\myangle{\mu_s^t}_\n(\alpha_1\oplus\tilde\alpha_1
  \oplus\cdots\oplus\tilde\alpha_d)_\n\\
&=& \myangle{\mu_{s,j}}_\n\cdot\myangle{\mu_s^t}_\n\cdot
  (\cdots \oplus 1 \oplus\tilde\gamma_j^{(-t)} \oplus 1
  \oplus\cdots)_\n\cdot
  (\alpha_1\oplus\tilde\alpha_1\oplus\cdots\oplus\tilde\alpha_d)_\n\\
&=& \myangle{\mu_{s,j}\mu_s^t}_\n
  (\alpha_1\oplus \tilde\alpha_1\oplus\cdots\oplus\tilde\alpha_{j-1}
  \oplus\tilde\gamma_j^{(-t)}\tilde\alpha_j\oplus\tilde\alpha_{j+1}
  \oplus\cdots\oplus\tilde\alpha_d)_\n.
\end{eqnarray*}
From the above equations, since $\mu_{s,j}\mu_s=\mu_s\mu_{s,j}$, we can see that
$\alpha\gamma=\gamma\alpha$ if and only if
$\tilde\alpha_j^{(-1)}\tilde\gamma_j=\tilde\gamma_j^{(-t)}\tilde\alpha_j$.
On the other hand,
\begin{eqnarray*}
\tilde\alpha_j^{(-1)}\tilde\gamma_j
&=& (\alpha_{j,s}\oplus\alpha_{j,1}\oplus\cdots\oplus\alpha_{j,s-1})_{\n_j}
  \cdot (\gamma_{j,1}\oplus\cdots\oplus\gamma_{j,s})_{\n_j}
  = \bigoplus_{\ell=1}^s \alpha_{j,\ell -1}\gamma_{j,\ell}, \\
\tilde\gamma_j^{(-t)}\tilde\alpha_j
&=& (\gamma_{j,1-t}\oplus\cdots\oplus\gamma_{j,s-t})_{\n_j}
  \cdot(\alpha_{j,1}\oplus\cdots\oplus\alpha_{j,s})_{\n_j}
  = \bigoplus_{\ell=1}^s \gamma_{j,\ell-t}\alpha_{j,\ell},
\end{eqnarray*}
where $\n_j = (n_{[j,1]}, n_{[j,2]},\ldots, n_{[j,s]})$.
By the above claim, we are done.

\medskip
(ii) \ \
As $\lk(\chi_1)=0$, we have
$\lk(\chi)=\sum_{i=2}^r n_i\lk_i(\chi_0)$.
By (i), for each $2\le i\le r$, there exists
a $P$-straight $n$-braid $\zeta_{i}$
such that $\lk(\zeta_{i})= n_i$ and $\zeta_{i}$ commutes with $\alpha$.
Let $\gamma_{i} = \zeta_{i}^{-\lk_i(\chi_0)}$,
then $\gamma_{i}$ is a $P$-straight $n$-braid such that
it commutes with $\alpha$ and $\lk(\gamma_{i})= -n_i\lk_i(\chi_0)$.
Let $\gamma=\gamma_{2}\gamma_{3}\cdots\gamma_{r}$.
Then $\gamma$ is a $P$-straight $n$-braid such that $\gamma\alpha =\alpha\gamma$.
Moreover, $\lk(\gamma)= - \sum_{i=2}^r n_i\lk_i(\chi_0)= -\lk(\chi)$.
\end{proof}

Now we are ready to prove Theorem~\ref{thm:main}.

\begin{proof}[\textbf{Proof of Theorem~\ref{thm:main}}]
We will show the theorem by induction on the braid index $n$.
If $n=2$, Theorem~\ref{thm:main} is obvious as we have observed
in the proof of Lemma~\ref{lem:same-ext}.
Suppose that $n>2$ and that the theorem is true
for braids with less than $n$ strands.

Recall that $1\in P\subset\{1,\ldots,n\}$, and that
$\alpha$ and $\beta$ are $P$-pure $n$-braids
such that $\alpha^k=\beta^k$ for some $k\ne 0$.
If $\alpha$ is either pseudo-Anosov or periodic, the theorem is true
by Lemma~\ref{lem:irred}.
Thus we assume that $\alpha$ is reducible and non-periodic.

\smallskip

If $\Rext(\alpha)$ is not standard, choose $\zeta\in B_{n,1}$ such that
$\zeta*\Rext(\alpha)=\Rext(\zeta\alpha\zeta^{-1})$ is standard.
By Lemma~\ref{lem:conj}, it suffices to prove the theorem
for $(\zeta\alpha\zeta^{-1},\zeta\beta\zeta^{-1},\pi_{\zeta}(P), k)$.
Therefore, without loss of generality,
we assume that $\Rext(\alpha)$ is standard.

\smallskip
There exists a composition $\n=(n_1,\ldots,n_r)$ of $n$ such that
$\Rext(\alpha)=\Rext(\beta)=\C_\n$,
hence $\alpha$ and $\beta$ are expressed as
$$
\alpha = \myangle{\alpha_0}_{\n}(\alpha_1\oplus\alpha_2\oplus\cdots\oplus\alpha_r)_{\n}
\quad\mbox{and}\quad
\beta = \myangle{\beta_0}_{\n} ( \beta_1\oplus\beta_2\oplus\cdots\oplus\beta_r)_{\n}.
$$

Since $\alpha$ and $\beta$ are $P$-pure,
$\alpha_i$ and $\beta_i$ are $P_{\n,i}$-pure for $i=0,1$ by Lemma~\ref{lem:indu}.
In particular, because $\alpha$ and $\beta$ are 1-pure, the braids
$\alpha_0$ and $\beta_0$ are 1-pure.
Hence
$$
\alpha^k = \myangle{\alpha_0^k}_{\n}(\alpha_1^k\oplus\cdots)_{\n}
\quad\mbox{and}\quad
\beta^k = \myangle{\beta_0^k}_{\n}(\beta_1^k\oplus\cdots)_{\n}.
$$
(Here the second interior braid of $\alpha^k$ is not necessarily $\alpha_2^k$
unlike the first interior braid $\alpha_1^k$.)
Since $\alpha^k = \beta^k$, we have $\alpha_0^k = \beta_0^k$ and $\alpha_1^k = \beta_1^k$.

Note that $\alpha_0$ is periodic or pseudo-Anosov.
If $\alpha_0$ is central, then it is obvious that $\alpha_0=\beta_0$.
If $\alpha_0$ is pseudo-Anosov, then $\alpha_0=\beta_0$
by Lemma~\ref{lem:NT}.
For these two cases, we are done by Lemma~\ref{lem:same-ext}.
Therefore we assume that $\alpha_0$ is periodic and non-central.

\medskip

Since $\alpha^k = \beta^k$, there exists $\zeta\in B_n$
with $\beta=\zeta\alpha\zeta^{-1}$ by Theorem~\ref{thm:gon}.
Notice that
$$
\Rext(\alpha)
=\Rext(\alpha^k)
=\Rext(\beta^k)
= \Rext(\beta)
= \Rext(\zeta\alpha\zeta^{-1})
= \zeta*\Rext(\alpha),
$$
i.e. $\zeta$ preserves the curve system $\Rext(\alpha)=\C_\n$.
Hence $\zeta$ can be expressed as
$$
\zeta = \myangle{\zeta_0 }_{\n} ( \zeta_1\oplus\zeta_2\oplus\cdots\oplus\zeta_r)_\n.
$$
We will replace $\zeta_1$ in the above expression of $\zeta$
with another braid $\xi_1$ in order to make it $P$-straight,
and then will multiply it by another $n$-braid $\xi'$ in order to make it
1-unlinked.

Because $\alpha_0$ and $\beta_0$ are 1-pure, periodic and non-central,
they have the first strand as the only pure strand by Corollary~\ref{cor:per}.
Hence the $r$-braid $\zeta_0$ must be 1-pure because
$\beta_0= \zeta_0 \alpha_0 \zeta_0^{-1}$.
In addition, $P_{\n,0} = \{ 1\}$ and $P=P_{\n,1}$.
Recall that $\alpha_1$ and $\beta_1$ are $P_{\n,1}$-pure.
Because $\alpha_1^k = \beta_1^k$ and $1\in P_{\n,1}$,
there exists a $P_{\n,1}$-straight, 1-unlinked $n_1$-braid $\xi_1$
such that $\beta_1=\xi_1\alpha_1\xi_1^{-1}$,
by the induction hypothesis on the braid index.
Let
$$
\xi= \myangle{\zeta_0}_{\n}
(\xi_1\oplus\zeta_2\oplus\zeta_3\oplus\cdots\oplus\zeta_r)_\n.
$$
Then $\xi$ is $P$-straight since $P=P_{\n,1}$, $\zeta_0$ is 1-pure
and $\xi_1$ is $P_{\n,1}$-straight.
Notice that $\zeta$ and $\xi$ are the same except for the first
interior braids, $\zeta_1$ and $\xi_1$.
Notice also that $\zeta_1\alpha_1\zeta_1^{-1}=\beta_1 = {\xi_1}\alpha_1{\xi_1}^{-1}$.
Therefore $\xi\alpha\xi^{-1} = \zeta\alpha\zeta^{-1}=\beta$.

By Lemma~\ref{lem:per-ext}~(ii),
there exists a $P$-straight $n$-braid $\xi'$ such that
$\xi'\alpha=\alpha\xi'$ and $\lk(\xi')=-\lk(\xi)$.
Let $\gamma=\xi\xi'$.
Then $\gamma$ is $P$-straight and 1-unlinked, and
$\gamma\alpha \gamma^{-1} = \xi\alpha\xi^{-1} =\beta$.
\end{proof}

\subsection*{Acknowledgements}

This work was done partially while the authors were visiting the Institute for
Mathematical Sciences, National University of Singapore in 2007.
We thank the institute for supporting the visit.
This work was supported by the National Research Foundation of Korea
(NRF) grant funded by the Korea government (MEST)
(No.~2009-0063965).

\end{document}